\documentclass[11pt]{amsart}
\usepackage{amsfonts,latexsym, color}
\usepackage{amsmath}
\usepackage{amssymb}
 \font\tenmsb=msbm10 at 12pt \font\sevenmsb=msbm7 at 8pt \font\fivemsb=msbm5 at
6pt
\newfam\msbfam
\textfont\msbfam=\tenmsb \scriptfont\msbfam=\sevenmsb \scriptscriptfont\msbfam=\fivemsb

\def\R{{\mathbb R}}
\def\C{{\mathcal C}}
\def\B{{\mathbb B}}
\def\N{{\mathbb N}}

\def\Z{{\mathbb Z}}
\def\Ec{{\mathcal E}}
\begin{document}
\newcommand{\reset}{\setcounter{equation}{0}}

\newcommand{\beq}{\begin{equation}}
\newcommand{\noi}{\noindent}
\newcommand{\eeq}{\end{equation}}
\newcommand{\dis}{\displaystyle}
\newcommand{\mint}{-\!\!\!\!\!\!\int}

\def \theequation{\arabic{section}.\arabic{equation}}

\newtheorem{thm}{Theorem}[section]
\newtheorem{lem}[thm]{Lemma}
\newtheorem{cor}[thm]{Corollary}
\newtheorem{prop}[thm]{Proposition}
\theoremstyle{definition}
\newtheorem{defn}[thm]{Definition}
\newtheorem{rem}[thm]{Remark}
\newcommand{\Td}{T^d}
\newcommand{\dd}{\,\mathrm d}
\newcommand{\eps}{\varepsilon}
\newcommand{\Ent}{\operatorname{Ent}}
\newcommand{\divergence}{\operatorname{div}}

\def \bx{\hspace{2.5mm}\rule{2.5mm}{2.5mm}} \def \vs{\vspace*{0.2cm}} \def
\hs{\hspace*{0.6cm}}
\def \ds{\displaystyle}
\def \p{\partial}
\def \O{\Omega}
\def \H{{\mathbb H}}
\def \b{\beta}
\def \m{\mu}
\def \T{{\mathbb T}}
\def \ou{{\overline u}}
\def \ov{{\overline v}}
\def \D{\Delta}
\def \l{\lambda}
\def \s{\sigma}
\def \e{\varepsilon}
\def \a{\alpha}
\def \o{\omega}
\def \b{\beta}
\def \wv{{\widetilde v}}
\def \hw{{\widehat w}}
\def \HH{{\mathcal H}}
\def \E{{\mathbb E}}

\def\cqfd{%
\mbox{ }%
\nolinebreak%
\hfill%
\rule{2mm} {2mm}%
\medbreak%
\par%
}
\def \pr {\noindent {\bf Proof.} }
\def \rmk {\noindent {\it Remark} }
\def \esp {\hspace{4mm}}
\def \dsp {\hspace{2mm}}
\def \ssp {\hspace{1mm}}

\title{Asymptotically sharp Hardy-Rellich inequalities on lattices}

\author{Xia Huang}
\address{School of Mathematical Sciences, East China Normal
University and Shanghai Key Laboratory of PMMPSchool of Mathematical Sciences, Key Laboratory of MEA(Ministry of Education) and Shanghai Key Laboratory of PMMP, East China Normal University, Shanghai, 200241, China}
\email{xhuang@cpde.ecnu.edu.cn}
\author{Dong Ye}
\address{School of Mathematical Sciences, Key Laboratory of MEA(Ministry of Education) and Shanghai Key Laboratory of PMMP, East China Normal
University, Shanghai, 200241, China}
\email{dye@math.ecnu.edu.cn}

\date{}
\begin{abstract}
We determine the sharp asymptotic behavior of the optimal constants in the discrete Hardy-Rellich inequalities on the lattice $\mathbb{Z}^d$. For every fixed integer $m\ge 1$, let ${\mathcal C}_{m,d}$ be the best constant in the $m$-th order Hardy-Rellich inequality. We prove that
$$\lim_{d\to\infty}\frac{\mathcal C_{m,d}}{d^m}=2^m.$$
Our approach combines a Fourier reduction to weighted inequalities with the flat torus and general weighted Hardy-Rellich identities of first and second order. A key novelty is the use of probabilistic concentration estimates, specifically Hoeffding's inequality and entropy methods, to handle estimates involving the anisotropic weight $\omega^\gamma$ $(\gamma\geq 1)$ where
$$\omega(x)=\sum_{j=1}^{d} \Big(\sin\frac{x_j}{2}\Big)^2$$ in a dimension-uniform manner. These tools yield asymptotically sharp weighted estimates on the torus, from which the lattice inequalities follow by iteration.
\end{abstract}
\maketitle

\medskip


\medskip


\section{Introduction}
In 1921, in a letter to Hardy, Landau \cite{L}  gave a proof of the following inequality with optimal constants: for any $p>1$ and any non-negative real sequence ${\{a_n}\}_{n\geq 1}$, we have
\begin{align}\label{L}
\left(\frac{p-1}{p}\right)^p\sum_{n=1}^\infty\left(\frac{a_1+a_2+\cdot\cdot\cdot+a_n}{n}\right)^p\leq \sum_{n=1}^\infty a^p_n.
\end{align}

Equivalently, let $C_c(\N)$ be the space of finitely supported real sequences and $\Ec(\N)$ the subspace of $u \in C_c(\N)$ with $u_0=0$. If we define $\nabla u_n:=u_n-u_{n-1}$ for $n\geq 1$, then
\begin{align*}
\sum_{n=1}^\infty |\nabla u_n|^p \geq \left(\frac{p-1}{p}\right)^p \sum_{n=1}^\infty \frac{|u_n|^p}{n^p},\quad \forall~u\in \Ec(\N).
\end{align*}
The most well-known case $p=2$ was established by Hardy \cite{H} when he proposed a simple proof of Hilbert's inequality for double sequences. Over the past century, Hardy type inequalities have since been extensively developed and played important roles in many branches of analysis and geometry. A vast literature exists, especially in the continuous setting.

\medskip
Although Hardy's inequalities have discrete origins, sharp discrete Hardy-Rellich inequalities remain much less understood than their continuous counterparts. For example, on the half-line $\N$, sharp discrete Hardy inequalities have been understood only recently. Keller-Pinchover-Pogorzelski \cite{KPP0, KPP} established the following optimal first order discrete Hardy inequality:
\begin{align}
\sum_{n=1}^\infty |\nabla u_n|^2 \geq \sum_{n=1}^\infty \omega_n u^2_n,\quad \forall~u\in \Ec(\N)
\end{align}
with critical weights $\omega_n=\frac{\Delta \sqrt n}{\sqrt n} > \frac{1}{4n^2}$ for all $n \ge 1$. Here $\Delta v_n=2v_n-v_{n+1}-v_{n-1}$ denotes the discrete Laplacian on $\N$. 

\smallskip
Higher-order analogues on $\N$ have also been studied recently. In particular, the authors \cite{HY1} determined the sharp constants for Hardy-Rellich inequalities of arbitrary order $\ell\geq 2$ on $\N$ as follows. For any $\ell \ge 1$, let ${\mathcal D}^\ell = \Delta^{\ell/2}$ for even $\ell$ and ${\mathcal D}^\ell = \nabla\Delta^{(\ell-1)/2}$ for odd $\ell$. The following inequality holds
\begin{align}\label{N}
\sum_{2n\geq \ell}|{\mathcal D}^\ell u_n|^2 \geq \left[\frac{(2\ell)!}{4^\ell \ell!}\right]^2 \sum_{n=\ell}^\infty \frac{u^2_n}{n^{2\ell}},\quad \forall\; u\in C_c(\N),~u_k=0,~0\leq k\leq \ell-1.
\end{align}
It is worth mentioning that the above constraint $u_k=0$ for $0\leq k\leq \ell-1$ is necessary for the $\ell$-th order Rellich inequality to hold on $\N$, and the constant in \eqref{N} is sharp and coincides with the sharp constant for the Hardy-Rellich inequality in the continuum, see for instance \cite{O}:
\begin{align*}
\int_0^\infty |\varphi^{(\ell)}(x)|^2 dx \geq \left[\frac{(2\ell)!}{4^\ell \ell!}\right]^2 \int_0^\infty \frac{|\varphi(x)|^2}{x^{2\ell}} dx, \quad \forall \;\varphi \in C^{\ell}_c(0,\infty).
\end{align*}
%Here $D^\ell = \Delta^{\ell/2}$ for even $\ell$ and $D^\ell = \nabla\Delta^{(\ell-1)/2}$ for odd $\ell$.

Very recently, Stampach-Waclawek \cite{SW} improved the inequalities \eqref{N} by establishing optimal weights for general $\ell \ge 1$.
\begin{align*}
\sum_{2n\geq \ell}|{\mathcal D}^\ell u_n|^2 \geq \sum_{n=\ell}^\infty \rho_n |u_n|^2,\quad u\in C_c(\N),~u_k=0,~0\leq k\leq \ell-1
\end{align*}
where the weights $\rho_n:=\frac{\Delta^\ell g_{\ell, n}}{g_{\ell, n}}$ with $g_{\ell, n}=\sqrt n \prod_{1\le j \le \ell-1}(n-j)$ are sharp in the sense of \cite{KPP}.

\medskip
Naturally, we turn to the higher-dimensional lattices $\Z^d$ with $d \ge 2$. However, the situation changes drastically, the question of the best constants of Hardy-Rellich inequalities is wide open for $\mathbb{Z}^d$ with $d > 2m$. More precisely, let $\Ec(\Z^d) = \{u \in C_c(\Z^d), u(0) = 0\}$ be the space of compactly supported real functions on $\Z^d$ ($d \in \N$) vanishing at the origin. We are interested in the following optimal Hardy-Rellich inequalities on the lattice $\Z^d$: 
\begin{align*}
\sum_{n\in\mathbb{Z}^d} |\nabla u(n)|^2 \geq \C_H(d)\sum_{n\in\mathbb{Z}^d\setminus{\{0}\}}\frac{|u(n)|^2}{|n|^2}, \quad \forall\; u \in \Ec(\Z^d)
\end{align*}
and
\begin{align*}
\sum_{n\in\mathbb{Z}^d} |\Delta u(n)|^2 \geq \C_R(d)\sum_{n\in\mathbb{Z}^d\setminus{\{0}\}}\frac{|u(n)|^2}{|n|^4}, \quad \forall\; u \in \Ec(\Z^d).
\end{align*}
Here $\nabla u$ and $\Delta u$ stand for the discrete gradient and the discrete Laplacian on $\Z^d$ respectively, that is,  
\begin{align*}
\nabla u(n) := \big(u(n)-u(n-e_j)\big), \quad \Delta u(n):=\sum_{j=1}^d \big[2 u(n)-u(n-e_j)-u(n+e_j)\big]
\end{align*}
where $\{e_j\}$ is the canonical basis of $\R^d$. More generally, for $m \ge 1$, let $\C_{m, d}$ be the best constant in the following inequality
\begin{align}\label{d}
\sum_{n\in \mathbb{Z}^d} |\mathcal{D}^m u(n)|^2\geq \C_{m, d} \sum_{n\in\mathbb{Z}^d\setminus{\{0}\}} \frac{|u(n)|^2}{|n|^{2m}}, \quad \forall\; u \in \Ec(\Z^d).
\end{align}
where
\begin{align*}
\mathcal{D}^m u=\begin{cases}
\Delta^{m/2}u, & \mbox{if $m$ is even},\\
\nabla\Delta^{(m-1)/2}u,& \mbox{if $m$ is odd}.
\end{cases}
\end{align*}
Obviously, $\C_H(d) = \C_{1, d}$ and $\C_R(d) = \C_{2, d}$.

\medskip
Actually, Determining the optimal constants in \eqref{d} explicitly appears to be difficult, even for $\C_H(d)$, $d \ge 3$. One major difference between the lattice and Euclidean settings is that many crucial tools in the continuum, such as the scaling argument or integration by parts, are no longer available. In \cite{G}, Gupta used Fourier series methods which transform the discrete situation to weighted Hardy-Rellich inequalities on the torus $\T^d$, see Lemma \ref{g} below; the same approach can also be found in \cite{AB,KL}. He showed a very striking phenomenon in \cite{G}: For each fixed order $m \ge 1$, the sharp constants $\mathcal{C}_{m,d}$ in \eqref{d} satisfy
\begin{align*}
0 < \liminf_{d\to \infty}\frac{\mathcal{C}_{m,d}}{d^m} \leq \limsup_{d\to \infty}\frac{\mathcal{C}_{m,d}}{d^m} < \infty.
\end{align*}
In other words, the optimal constants $\mathcal{C}_{m,d}$ is of order $d^m$ as the dimension $d$ goes to infinity, which is in stark contrast with the continuous situation, since the corresponding sharp constants in $\R^d$ grow like $d^{2m}$, see \cite{TZ, B, GM, HY2}. For comparison, with $m = 2\ell$, we have the following optimal Rellich estimates in $\R^d$
\begin{align*}
\int_{\R^d} |\Delta^\ell u|^2 dx \geq \prod_{k=0}^{\ell-1} \frac{(d+4k)^2(d-4-4k)^2}{16}\int_{\R^d} \frac{u^2}{|x|^{4\ell}}dx,\quad \forall~u\in C_c^{2m}(\R^d).
\end{align*} 

\medskip
Gupta's result shows that the sharp constants $\mathcal{C}_{m,d}$ for $\Z^d$ in \eqref{d} are very different from those in $\R^d$ when $d$ becomes large. 
A natural and important question is then to identify the exact asymptotic behavior of $\mathcal{C}_{m,d}$ as $d\to\infty$. The main result of this paper gives a complete answer and resolves a question left open in \cite{G}.
\begin{thm}
\label{main}
Let $m\geq 1$ be an integer and $\C_{m, d}$ be the sharp constant in the Hardy-Rellich inequality \eqref{d}, then
$$\lim_{d\to\infty}\frac{\C_{m, d}}{d^m} =2^m.$$  
\end{thm}

Following \cite{G}, we will convert the discrete estimates on $\Z^d$ into some equivalent integral inequalities on the flat torus $\T^d$. Since we always consider $2\pi$-{\bf periodic} functions, with a slight abuse of notation, we also denote $\mathbb{T}^d=[{-\pi},\pi]^d$ and define
\begin{align*}
C^k(\mathbb{T}^d)= \Big\{\varphi \in C^k(\mathbb{R}^d)\big|\;\varphi(x+2\pi e_j)=\varphi(x),~\forall ~x\in\mathbb{R}^d,~1\leq j\leq d\Big\}.
\end{align*}
We denote by $\dot{C^k}(\mathbb{T}^d)$ the subspace of $C^k(\mathbb{T}^d)$ with zero average, i.e.
\begin{align*}
\dot{C^k}(\mathbb{T}^d)= \Big\{\varphi \in C^k(\mathbb{T}^d),\; \overline \varphi := \int_{\mathbb{T}^d} \varphi dx=0\Big\}.
\end{align*}

\medskip
Using Fourier analysis, we have the following lemma, see Lemmas 2.2 and 2.3 in \cite{G} for instance. Let ${\mathcal D}^\ell$ denote the discrete $\ell$-th order differential operator defined previously, and $D^\ell$ denote the continuous $\ell$-th order differential operator in $\R^d$, that is, $D^\ell =\Delta^{\ell/2}$, if $\ell$ is even and $D^\ell =\nabla\Delta^{(\ell-1)/2}$ if $\ell$ is odd.
\begin{lem}\label{g}
Let $u\in \Ec(\mathbb{Z}^d)$ and $m \in \N$, then there exists $\Psi\in \dot C^\infty(\mathbb{T}^d)$ such that
\begin{align*}
\sum_{n\in\mathbb{Z}^d\setminus{\{0}\}} \frac{|u(n)|^2}{|n|^{2m}}=\int_{\mathbb{T}^d} |D^m\Psi(x)|^2 dx
\end{align*}
and
\begin{align*}
\sum_{n\in\mathbb{Z}^d} |\mathcal{D}^m u(n)|^2 = 4^{m}\int_{\mathbb{T}^d} |D^{2m}\Psi(x)|^2\omega^m dx,\quad \text{with}~\omega(x):= \sum_{j=1}^d {\Big(\sin\frac{x_j}{2}\Big)}^2.
\end{align*}
\end{lem}

Consequently, to prove Theorem \ref{main}, we need only to consider the corresponding estimates over $\dot{C}^k(\T^d)$. More generally, we determine the asymptotic behaviors as follows.
\begin{thm}\label{th3}
Let $\ell, k $ be nonnegative integers with $0 < k - \ell \le \gamma \in \N$. Let $C_{d, k, \ell, \gamma}$ be the best constant for the following weighted Hardy-Rellich inequalities: 
\begin{align*}
  \int_{\T^d} |D^k\phi|^2 \omega^\gamma dx \ge C_{d, k,\ell, \gamma} \int_{\T^d} |D^{\ell}\phi|^2 \omega^{\gamma -k + \ell}dx, \quad \mbox{ for all } \phi \in \dot{C}^k(\T^d).
\end{align*}
Then
\begin{align}\label{mainc}
\lim_{d\to\infty} \frac{C_{d, k, \ell, \gamma}}{d^{k-\ell}} = 2^{\ell-k}.
\end{align}
\end{thm}

To handle \eqref{mainc}, the main technical challenge lies in the anisotropic nature of the weights $\omega^\gamma$, and we need to manage the strong coupling between different directions. The crucial arguments are the following sharp estimates.
\begin{prop}\label{lower12}
Let $\gamma$ be a fixed positive integer. The following estimates hold
\begin{align}
\label{gap}
\int_{\T^d}|\nabla \varphi|^2 \omega^\gamma dx \ge \Big(1 - o_d(1)\Big) \int_{\T^d} \varphi^2 \omega^\gamma dx, \quad \forall\; \varphi \in \dot C^1(\T^d);
\end{align}
\begin{align}\label{general1}
  \int_{\T^d} |\nabla\varphi|^2 \omega^{\gamma} dx \ge \frac{d}{2}\Big(1 - o_d(1)\Big) \int_{\T^d} \varphi^2 \omega^{\gamma-1}dx, \quad \forall\;  \varphi \in \dot{C}^1(\T^d);
\end{align}
and 
\begin{align}
\label{gap1}
\int_{\T^d} |\Delta \varphi|^2 \omega^\gamma dx \ge \Big(1 - o_d(1)\Big) \int_{\T^d} |\nabla\varphi|^2 \omega^\gamma dx,\quad \forall\; \varphi\in\dot C^2(\T^d);
\end{align}
\begin{align}\label{general2}
  \int_{\T^d} |\Delta\varphi|^2 \omega^{\gamma} dx \ge \frac{d}{2}\Big(1 - o_d(1)\Big) \int_{\T^d} |\nabla\varphi|^2 \omega^{\gamma-1}dx, \quad \forall\;  \varphi \in \dot C^2(\T^d).
\end{align}
\end{prop}
Throughout the paper, $o_d(1)$ denotes a quantity tending to zero as $d$ goes to infinity, depending only on the fixed order parameters such as $\gamma,~k,~\ell$, but independent of the test functions.

\medskip

The estimates \eqref{gap} and \eqref{gap1} are elementary when $\gamma = 0$ as we know the spectrum of $-\Delta$ on the flat torus, but the weighted versions with $\gamma \ge 1$ require new arguments. Due to the anisotropy and strong coupling of the weights $\omega^\gamma$, the standard Bessel-pair approach fails to provide sharp estimates. 

Our strategy consists of two novel ingredients.
\begin{itemize}
\item We apply general Hardy-Rellich identities (see Lemma 2.1), including a new second order Rellich-type identity that couples the Hessian with a vector field via tensor
products, see \eqref{hardy2F}. These identities are crucial for handling the error terms produced by the anisotropic weights.
\item We use the probabilistic concentration estimates that exploit the interpretation of 
$$\rho := \frac{2\omega}{d} = \frac{1}{d}\sum_{j=1}^d (1-\cos x_j)$$
    as the empirical mean of i.i.d. random variables. This allows us to control the degenerate regions via Hoeffding's inequality as $d \to \infty$; and to handle entropy estimates uniformly in $d$ using the dimension-independent logarithmic Sobolev inequality on $\T^d$. See Lemma \ref{lemHoeff} and estimate \eqref{lemEnt}.
\end{itemize}

\medskip
The paper is organized as follows. 
Section 2 collects the weighted identities and probabilistic tools, including Hoeffding's estimate and the entropy inequality. 
Section 3 proves the fundamental weighted estimates on \(\T^d\) in Proposition \ref{lower12}. 
Section 4 establishes the asymptotically sharp torus inequalities and completes the proof of the main theorem for \(\Z^d\) by lifting the torus estimates back to the lattice.
\section{Preliminaries}
\reset

Here we first introduce two general identities with general weights. They will play an important role in estimating the lower-order terms produced by the anisotropic weights $\omega^\gamma$.
\begin{lem}
\label{HardyF}
Let $V\in C^1(\T^d)$ and $\vec F\in C^1(\mathbb{T}^d,\mathbb{R}^d)$. Assume that $\varphi \in  C^1(T^d)$, then
\begin{align}
\label{hardy1F}
\int_{\mathbb{T}^d}V |\nabla \varphi|^2 dx = \int_{\mathbb{T}^d} \Big[{\rm div}(V\vec F)-V |\vec F|^2\Big]\varphi^2 dx + \int_{\mathbb{T}^d} V\big|\nabla \varphi + \varphi\vec F \big|^2 dx.
\end{align}
Now let $ V\in C^2(\mathbb{T}^d)$, then for any $\varphi\in  C^2(\mathbb{T}^d)$, there holds
\begin{align}
\label{hardy2F}
\begin{split}
\int_{\mathbb{T}^d}V |\Delta \varphi|^2 dx & = \int_{\mathbb{T}^d} \Big[{\rm div}(V\vec F)-V|\vec F|^2-\Delta V\Big]|\nabla \varphi|^2 dx +\sum_{i,j=1}^d\int_{\mathbb{T}^d}V_{x_ix_j}\varphi_{x_i}\varphi_{x_j}dx\\
&\quad + \int_{\mathbb{T}^d}V \big|\nabla^2 \varphi + \vec F\otimes \nabla \varphi\big|^2 dx.
\end{split}
\end{align}
\end{lem} 

In particular, let $\vec F=-\frac{\nabla f}{f}$ and assume that supp$(\varphi) \subset\subset \{f \ne 0\}$, we obtain 
\begin{align}
\label{hardy1}
\int_{\mathbb{T}^d}V |\nabla \varphi|^2 dx = - \int_{\mathbb{T}^d} \frac{{\rm div}(V\nabla f)}{f} \varphi^2 dx + \int_{\mathbb{T}^d} V\Big|\nabla \varphi - \frac{\nabla f}{f} \varphi\Big|^2 dx
\end{align}
and
\begin{align}
\label{rellich}
\begin{split}
\int_{\mathbb{T}^d} V|\Delta \varphi|^2 dx & = -\int_{\mathbb{T}^d} \Big[\frac{{\rm div}(V\nabla f)}{f}+\Delta V\Big]|\nabla \varphi|^2 dx+\sum_{i,j=1}^d\int_{\mathbb{T}^d}V_{x_ix_j}\varphi_{x_i}\varphi_{x_j}dx \\
&\quad +\int_{\mathbb{T}^d} V\Big|\nabla^2 \varphi - \frac{\nabla f\otimes \nabla \varphi}{f}\Big|^2 dx.
\end{split}
\end{align}

\begin{rem}
The equalities \eqref{hardy2F} and \eqref{rellich} seem to be new Rellich-type identities, associating the first and second order energies through a tensor product coupling with a general vector field. Similar identities hold on a general domain $\Omega$ if we assume that $\varphi$ is compactly supported in $\Omega$.
\end{rem}

\medskip
\noindent
{\bf Proof.} The first order identities \eqref{hardy1F} and \eqref{hardy1} are somehow well known; they are obtained by developing the last integral in \eqref{hardy1F} and using integration by parts, where all boundary integrals cancel each other because of the periodicity assumption, see also \cite{HY}. So we omit the proof of \eqref{hardy1F} and \eqref{hardy1}.

\smallskip
Consider the second-order identity \eqref{hardy2F}. As $\nabla^2 \varphi = (\varphi_{x_ix_j})$ denotes the Hessian matrix of $\varphi$ and $\vec{F} \otimes \nabla\varphi$ is the tensor product matrix $(F_i  \varphi_{x_j})$,
\begin{align*}
\big|\nabla^{2} \varphi + \vec{F}\otimes \nabla  \varphi\big|^2 =\sum^d_{i,j=1} \big(\varphi_{x_ix_j}+ F_i  \varphi_{x_j}\big)^2=\sum_{i,j=1}^d  \varphi^2_{x_ix_j}+ 2\sum_{i,j=1}^d  \varphi_{x_ix_j} F_i  \varphi_{x_j}+\sum_{i,j=1}^d F^2_i  \varphi^2_{x_j}.
\end{align*}
Hence
\begin{align*}
\int_{\T^d} V |\nabla^{2} \varphi + \vec{F}\otimes \nabla  \varphi|^{2}dx & = \int_{\T^d}V|\nabla^{2} \varphi|^{2} dx + \int_{\T^d}V|\vec{F}|^{2}|\nabla  \varphi|^{2} dx+ \int_{\T^d}V\vec{F}\cdot\nabla \big(|\nabla  \varphi|^{2}\big)dx\\
 & = \int_{\T^d}V|\nabla^2  \varphi|^{2} dx+ \int_{\T^d}\big[V |\vec F|^2 - {\rm {div}}(V\vec F)\big]|\nabla\varphi|^{2}dx.
%- \sum_{i,j = 1}^{n}\int_{\Omega}u_{x_{i}}u_{x_{j}}V_{x_{i}x_{j}}
%+ \int_{\Omega}\bigl(V|\vec{F}|^{2} - {\rm{div}}(V\vec{F})\bigr)|\nabla u|^{2}.
\end{align*}
On the other hand, we have the following weighted Bochner type identity
\begin{align*}
\int_{\T^d} V |\nabla^2  \varphi|^2dx = \int_{\T^d} V|\Delta  \varphi|^2 dx-\sum_{i,j} \int_{\T^d} V_{x_ix_j}  \varphi_{x_i}  \varphi_{x_j} dx  + \int_{\T^d} \Delta V |\nabla  \varphi|^2 dx.
\end{align*}
For completeness, we give a detailed proof here and notice that no boundary term appears due to the periodicity.
\begin{align*}
\int_{\T^d} V |\nabla^2  \varphi|^2dx & = \sum_{i,j}\int_{\T^d} V  \varphi^2_{x_ix_j} dx\\
& = -\sum_{i,j}\int_{\T^d} V_{x_i}  \varphi_{x_j}  \varphi_{x_ix_j} dx-\sum_{i,j}\int_{\T^d} V  \varphi_{x_j}  \varphi_{x_ix_ix_j}dx\\
& = -\sum_{i,j}\int_{\T^d} V_{x_i}  \varphi_{x_j}  \varphi_{x_ix_j} dx + \int_{\T^d} {\rm div}\big(V\nabla\varphi) \Delta\varphi dx\\
& = -\sum_{i,j}\int_{\T^d} V_{x_i}  \varphi_{x_j}  \varphi_{x_ix_j} dx + \int_{\T^d} V|\Delta  \varphi|^2 dx -\sum_{i,j}\int_{\T^d}  \varphi_{x_i}\big(V_{x_j}\varphi_{x_j}\big)_{x_i} dx\\
& = -2\sum_{i,j}\int_{\T^d} V_{x_i}  \varphi_{x_j}  \varphi_{x_ix_j} dx + \int_{\T^d} V|\Delta  \varphi|^2 dx -\sum_{i,j} \int_{\T^d} V_{x_ix_j}\varphi_{x_i}  \varphi_{x_j} dx\\
& = -\int_{\T^d} \nabla V\cdot\nabla\big(|\nabla \varphi|^2\big) dx + \int_{\T^d} V|\Delta  \varphi|^2 dx -\sum_{i,j} \int_{\T^d} V_{x_ix_j}\varphi_{x_i}  \varphi_{x_j} dx\\
& = \int_{\T^d} V|\Delta  \varphi|^2 dx-\sum_{i,j} \int_{\T^d} V_{x_ix_j}  \varphi_{x_i}  \varphi_{x_j} dx  + \int_{\T^d} \Delta V |\nabla  \varphi|^2 dx.
\end{align*}
Combining the above equalities, we readily obtain \eqref{hardy2F}.  The formula \eqref{rellich} follows by taking $\vec F=-\frac{\nabla f}{f}$. This completes the proof. \qed

\medskip
As mentioned, we also make use of some inequalities from probability theory. The first one is the following very special case of Hoeffding's estimate (see \cite{BLM} or \cite{Ho}).
\begin{lem}
\label{lemHoeff}
Let $(Y_i)_{1\le i\le d}$ be a family of independent random variables on the probability space $(\Omega, \mu)$, with $\E_\mu(Y_i) = 0$, and $a \le Y_i \le b$. 
Let $\overline Y = \frac{1}{d}\sum_i Y_i$, we have then
\begin{align}
\label{Hoefg}
P\big(|\overline Y| \ge t\big) \le 2e^{-\frac{2dt^2}{(b-a)^2}}, \quad \forall\; t > 0
\end{align}
and
\begin{align*}
\E_\mu\big(e^{s\overline Y}\big) \le e^{\frac{s^2(b-a)^2}{8d}},\quad \forall~s\in\R.
\end{align*}
In particular, if $Y_i\in[-1,1]$, there holds
\begin{align}
\label{Hoeff}
\E_\mu\big(e^{s\overline Y}\big) \le e^\frac{s^2}{2d}, \quad \forall\; s\in \R.
\end{align}
\end{lem}

Another useful estimate followa from the Donsker-Varadhan variational formula for entropy, see for instance \cite{Le}. With abuse of notation, we denote by $d\mu = (2\pi)^{-d}dx$ the standard uniform probability measure on the flat torus $\T^d$.
\begin{lem}
For any $f\in C^1(\T^d)$, any nonnegative continuous function $W_0$, and any $s>0$, there holds
\begin{align}
\label{ent}
\E_\mu(W_0f^2) \le
\frac1s\Ent_\mu(f^2)
+ \frac1s
\Big(\ln \int_{\T^d} e^{sW_0}d\mu\Big) \E_\mu(f^2).
\end{align}
\end{lem}

We recall also the following uniform entropy estimate given by Gross \cite{Gro}. 
\begin{align}
\label{lemEnt}
\Ent_\mu(f^2) \le C_{LS} \int_{\T^d} |\nabla f|^2 d\mu, \quad \forall\; f \in H^1(\T^d).
\end{align}
Notice that $C_{LS}>0$ is independent of $d$, and $\Ent$ stands for the entropy defined by 
\begin{align*}
\Ent_\mu(f^2)=
\int_{\T^d} f^2
\ln\frac{f^2}{\E_\mu(f^2)} d\mu.
\end{align*}
In other words, the best constant of the log-Sobolev inequality over the flat torus $\T^d$ is independent of the dimension $d$. 

\section{Key estimates}
\reset
In this section, we will prove the estimates in Proposition \ref{lower12}. 

\subsection{Proof of \eqref{gap}} The fundamental step is to establish \eqref{gap}. We denote $d\mu = (2\pi)^{-d}dx$ the probability measure on $\T^d$, and
$$\rho = \frac{2\omega}{d} = 1- \frac1d\sum_{j=1}^d \cos x_j.$$
Direct calculations give
\begin{align} \label{estomega}
|\nabla \omega(x)|^2=\omega(x)-\sum_{j=1}^d \Big(\sin\frac{x_j}{2}\Big)^4 \leq \omega-\frac{\omega^2}{d},\quad\Delta \omega(x)=\frac{d}{2}-\omega.
\end{align}
Hence we have
\begin{align}
\label{estrho}
\rho \in [0, 2], \quad |\nabla\rho|^2 \le \frac{2\rho-\rho^2}{d}, \quad \Delta\rho=1-\rho.
\end{align}

Let $\varphi \in \dot{C}^1(\T^d)$. Denote $v = h\varphi$ with $h = \rho^\beta$, $\beta = \frac{\gamma}{2} \ge \frac{1}{2}$. Applying \eqref{hardy1} with $V \equiv 1$ and $f = h$, there holds, for large $d$
\begin{align}
\label{new11}
\int_{\T^d} |\nabla v|^2 d\mu + \int_{\T^d} \frac{\Delta h}{h} v^2 d\mu = \int_{\T^d} \Big|\nabla v - \frac{\nabla h}{h} v\Big|^2 d\mu = \int_{\T^d} h^2|\nabla\varphi|^2 d\mu = \int_{\T^d} \rho^\gamma|\nabla\varphi|^2 d\mu
\end{align}

Notice that the weight $h$ is degenerate at $x = 0_{\R^d}$, but the identity is justified in the following standard way: We first use $h_\epsilon = (\rho +\epsilon)^\beta$ and $\epsilon > 0$; we get the corresponding equalities with $h_\epsilon$. As the test function is smooth and the possibly singular factors are integrable over $\T^d$ for $d \ge 3$, \eqref{new11} follows by the dominated convergence as $\epsilon\to 0$. 

\begin{rem}
In the sequel, we use several times the above approach, but we will not repeat the limit process. For the inequalities, we proceed also directly for the weights like $\rho^\gamma$, without mentioning eventual limit process using the dominated convergence or Fatou's lemma.
%In the same spirit, we omit to say that the computation work always with suitable large dimensions $d \ge d_0$.
\end{rem}

\medskip
Now we estimate the mean value of $v$, denoted by $\overline v = \E_\mu(v)$, we claim that
\begin{align}
\label{vbar}
\overline v^2 \le o_d(1)\int_{\T^d} v^2d\mu, \quad \mbox{as }\; d\to \infty.
\end{align}
Recall that $o_d(1)$ means always a sequence (independent of $\varphi$ or $v$) tending to zero as $d$ goes to infinity. As $\varphi \in \dot C^1(\T^d)$ and $v = \rho^\beta\varphi$,
\begin{align}
\label{new12}
\overline v^2 = \Big(\int_{\T^d} (1-\rho^{-\beta}) vd\mu\Big)^2 \leq \delta_d \int_{\T^d} v^2d\mu\quad \mbox{with } \; \delta_d := \int_{\T^d}\big(1-\rho^{-\beta}\big)^2d\mu.
\end{align}
On one hand, as $t \mapsto t^{-\beta}$ is Lipschitz over $[\frac{1}{2}, \infty)$,
\begin{align*}
\int_{\T^d\cap\{\rho\ge \frac{1}{2}\}}\big(1-\rho^{-\beta}\big)^2d\mu &\leq C_\gamma \int_{\T^d\cap\{\rho\ge \frac{1}{2}\}} \big(1-\rho\big)^2d\mu\\& \leq C_\gamma \int_{\T^d}\big(1-\rho\big)^2d\mu = C_\gamma {\rm Var}(\rho)=\frac{C_\gamma'}{d}.
\end{align*}
Here and after, $C_\gamma$ or $C_\gamma'$ denotes a positive constant depending on $\gamma$ but independent of the dimension $d$. 

\smallskip
On the other hand, since $\sin^2(\frac{t}{2})\geq \frac{t^2}{\pi^2}$ for $t\in[-\pi,\pi]$,
\begin{align*}
\rho = \frac{2\omega}{d} \geq \frac{2}{d}\sum_{j=1}^d \frac{|x_j|^2}{\pi^2} = \frac{2|x|^2}{\pi^2 d}, \quad \forall\; x \in \T^d.
\end{align*}
Let $|\B_d|$ be the volume of the Euclidean unit ball $\B_d \subset \R^d$. For $r > 0$, a direct calculation yields, as $d \to \infty$,
\begin{align*}
\frac{\mu\big(\T^d\cap\{\rho \le r\}\big)}{r^{d/2}} \le \frac{1}{(2\pi)^d}|\B_d| \Big(\frac{\pi^2d}{2}\Big)^\frac{d}{2} = \frac{(\pi d)^\frac{d}{2}}{2^\frac{3d}{2}\Gamma\big(\frac{d}{2} + 1\big)} \sim \frac{1}{\sqrt{\pi d}}\Big(\frac{\pi e}{4}\Big)^\frac{d}{2}.
\end{align*}
Hence there is $C_0 > 0$ such that for $d$ large and any $r > 0$, $\mu\big(\T^d\cap\{\rho \le r\}\big) \le (C_0r)^{d/2}$. 
Fix $M \in \mathbb{N}$ large such that $C_02^{-M}\leq 1$, we get, for large enough $d$,
\begin{align*}
\int_{\T^d\cap\{\rho < \frac{1}{2}\}}\big(1-\rho^{-\beta}\big)^2d\mu & = \sum_{k=2M}^\infty \int_{\T^d\cap\{2^{-k-1}\le \rho < 2^{-k}\}} \big(1-\rho^{-\beta}\big)^2 d\mu\\
& \quad + \int_{\T^d\cap\{2^{-2M}\le \rho < \frac{1}{2}\}} \big(1-\rho^{-\beta}\big)^2 d\mu\\
& \leq C_\gamma\sum_{k=2M}^\infty 2^{(k+1)\gamma}\mu\big(\T^d\cap\{\rho < 2^{-k}\}\big) + C_\gamma\mu\big(\T^d\cap\{\rho < 2^{-1}\}\big)\\
& \le C_\gamma\sum_{k=2M}^\infty 2^{(k+1)\gamma}2^{-kd/4} + C_\gamma\mu\big(\T^d\cap\{\rho < 2^{-1}\}\big).
\end{align*}
Clearly, the above series goes to zero when $d \to \infty$ by dominated convergence. For the second term, we consider $\rho$ as the average of a sequence of i.i.d. random variables $X_i = 1 - \cos x_i$ with $\E_\mu(X_i) = 1$. As $\rho = \overline Y + 1$ with $Y_i = X_i -1$, thanks to \eqref{Hoefg}, we have, 
 $$\mu\big(\T^d\cap\{\rho < 2^{-1}\}\big) \leq \mu\Big(|\overline Y| > \frac{1}{2}\Big) \rightarrow 0.$$
Thus $\lim_{d\to\infty} \delta_d = 0$. By \eqref{new12}, the claim \eqref{vbar} holds true. Using the standard Poincar\'e inequality over $\T^d$, there holds
\begin{align}
\label{new13}
\int_{\T^d} |\nabla v|^2 d\mu \ge \int_{\T^d} v^2 d\mu - \overline v^2 \ge (1 - \delta_d)\int_{\T^d} v^2 d\mu.
\end{align}

\medskip
Next, we claim
\begin{align}
\label{W}
W = \max\Big({-\frac{\Delta h}{h}}, 0\Big) \le C_\gamma \Big[(\rho-1)_+ + d^{-1}\Big]
\end{align}
In fact,
\begin{align*}
\frac{\Delta h}{h} = \beta\frac{\Delta\rho}{\rho} + \beta(\beta-1)\frac{|\nabla\rho|^2}{\rho^2}
= \beta\Big(\frac{1}{\rho}-1\Big) + \beta(\beta-1)\frac{|\nabla\rho|^2}{\rho^2}.
\end{align*}
Using \eqref{estrho},
\begin{itemize}
\item If $\beta \ge 1$ and $\rho \in (0, 1]$, then $W = 0$.
\item If $\beta \ge 1$ and $\rho \in (1, 2]$, then $W \le \beta\big(1 - \rho^{-1}\big) \le \beta (\rho -1)$.
%\item If $\frac{1}{2} \le \beta < 1$ and $\rho \in (0, 1]$, then $W \le \beta(1-\beta)\frac{|\nabla\rho|^2}{\rho^2} \le \frac{1}{2\rho d} \le d^{-1}$.
\item If $\frac{1}{2} \le \beta < 1$ and $\rho \in (0, 1]$, then $W \le \frac{2\beta(1-\beta)}{d}$ because
\begin{align*}
 \frac{\Delta h}{h} \ge \beta\Big(\frac{1}{\rho}-1\Big) + \beta(\beta -1)\frac{2}{\rho d} & = \beta\Big[\Big(1 - \frac{2- 2\beta}{d}\Big)\frac{1}{\rho} -1\Big]\\
 & \ge \beta\Big(1 - \frac{2 - 2\beta}{d} -1\Big) = -\frac{2\beta(1-\beta)}{d}.
\end{align*}
\item If $\frac{1}{2} \le \beta < 1$ and $\rho \in (1, 2]$, then $$W = \beta\Big(1 - \frac{1}{\rho}\Big) + \beta(1-\beta)\frac{|\nabla\rho|^2}{\rho^2} \le \beta(\rho-1) + \beta(1-\beta)\frac{2}{\rho d} \le C_\gamma \Big[(\rho-1)_+ + d^{-1}\Big].$$
\end{itemize}
So \eqref{W} holds true.

\medskip
Consider again $Y_i = X_i -1$, the sequence of independent random variables, applying Hoeffding's estimate \eqref{Hoeff} with $s = A\sqrt{d}$, and a fixed $A > 0$, there holds
\begin{align}
\label{Hoeff1}
\E_\mu\big(e^{A\sqrt{d}(\rho-1)_+}\big) \le \E_\mu\big(e^{A\sqrt{d}(\rho-1)}+1\big) \le e^\frac{A^2}{2}+1.
\end{align}
Now we can estimate the integral of $Wv^2$. Take $s = \sqrt{d}$ and $W_0 = C_\gamma[(\rho -1)_+ + d^{-1}]$, $f = v$. Combining \eqref{ent} and \eqref{Hoeff1}, we have 
\begin{align*}
\int_{\T^d} Wv^2 d\mu \leq \int_{\T^d} W_0v^2 d\mu &\leq \frac{1}{\sqrt d}\Big[\Ent_\mu(v^2) + \Big(\ln \int_{\T^d} e^{\sqrt{d}W_0}d\mu\Big)\E_\mu(v^2)\Big]\\
& \le \frac{1}{\sqrt d}\Big[C_{\rm LS} \int_{\T^d} |\nabla v|^2 d\mu + C_\gamma\E_\mu(v^2)\Big]\\
& \le \frac{C_\gamma}{\sqrt d}\Big[\int_{\T^d} |\nabla v|^2 d\mu + \int_{\T^d} v^2 d\mu\Big].
\end{align*}

Finally, combining with \eqref{new11} and \eqref{new13},
\begin{align*}
\int_{\T^d} \rho^\gamma|\nabla\varphi|^2 d\mu = \int_{\T^d} |\nabla v|^2 d\mu + \int_{\T^d} \frac{\Delta h}{h} v^2 d\mu & \ge \int_{\T^d} |\nabla v|^2 d\mu - \int_{\T^d} W v^2 d\mu\\
& \ge \Big(1 - \frac{C_\gamma}{\sqrt d}\Big)\int_{\T^d} |\nabla v|^2 d\mu  - \frac{C_\gamma}{\sqrt d}\int_{\T^d} v^2 d\mu\\
& \ge \Big[\Big(1 - \frac{C_\gamma}{\sqrt d}\Big)(1 - \delta_d) - \frac{C_\gamma}{\sqrt d}\Big]\int_{\T^d} v^2 d\mu\\
& = \Big(1 - o_d(1)\Big)\int_{\T^d} \rho^\gamma \varphi^2 d\mu.
\end{align*}
This completes the proof of \eqref{gap}. \qed

\subsection{Proof of \eqref{general1}}
Here we will apply the Hardy identity \eqref{hardy1}. Let $V = \omega^\gamma$ and $f = \omega^\alpha$, for $\alpha(\alpha+ \gamma -1) \ge 0$, a direct computation and \eqref{estomega} yield
\begin{align*}
-\frac{{\rm div}(V\nabla f)}{f} & = -\frac{\alpha(\alpha+ \gamma -1)\omega^{\alpha + \gamma - 2}|\nabla \omega|^2 + \alpha\omega^{\alpha+\gamma-1}\Delta \omega}{\omega^\alpha}\\
& \geq -\alpha\Big(\alpha + \gamma - 1 +\frac{d}{2}\Big)\omega^{\gamma -1} + \alpha\omega^{\gamma}.
\end{align*}
The singularity at the origin is again negligible. Therefore, applying \eqref{hardy1} and \eqref{gap} with $\alpha \le 0$ and $d \ge 2(\gamma-1-\alpha)$, for any $\varphi \in \dot{C}^1(\T^d)$,
\begin{align*}
\int_{\mathbb{T}^d} \omega^\gamma |\nabla \varphi|^2 dx & \ge -\alpha\big(\alpha + \gamma - 1 +\frac{d}{2}\big)\int_{\T^d} \omega^{\gamma -1}\varphi^2 dx + \alpha\int_{\T^d}\omega^{\gamma}\varphi^2 dx\\
& \ge -\alpha\Big(\alpha + \gamma - 1 +\frac{d}{2}\Big)\int_{\T^d} \omega^{\gamma -1}\varphi^2 dx + \alpha\big(1 + \epsilon_d\big)\int_{\T^d}\omega^{\gamma}|\nabla \varphi|^2 dx.
\end{align*}
Therefore, we obtain
\begin{align}
\int_{\mathbb{T}^d} \omega^\gamma |\nabla \varphi|^2 dx & \ge -\alpha\Big(\alpha + \gamma - 1 +\frac{d}{2}\Big)\frac{1}{1 - \alpha\big(1 + o_d(1)\big)} \int_{\T^d} \omega^{\gamma -1}\varphi^2 dx.
\end{align}
For fixed $\gamma$, choosing $\alpha = {-\sqrt{d}}$. Then $\alpha(\alpha+\gamma-1)\geq 0$ for large $d$, and
\begin{align*}
H(-\sqrt d)=\frac{\sqrt d\left(\frac{d}{2}-\sqrt d+\gamma-1\right)}{1+\sqrt d(1+o_d(1))}=\frac{d}{2}\Big(1-o_d(1)\Big),
\end{align*}
we conclude readily \eqref{general1}. \qed

\medskip

\subsection{Proof of \eqref{gap1}}
Firstly, for $\varphi \in \dot C^2(\T^d)$,
\begin{align}
\label{new1}
\begin{split}
\int_{\mathbb{T}^d} \omega^\gamma|\nabla\varphi|^2 dx & = \frac{1}{2}\int_{\T^d}\Delta(\omega^\gamma) \varphi^2 dx -\int_{\mathbb{T}^d}\omega^\gamma \varphi\Delta \varphi dx\\
& \leq \frac{1}{2}\int_{\T^d}\Delta(\omega^\gamma) \varphi^2 dx + \frac{1}{2}\int_{\mathbb{T}^d}\omega^\gamma\varphi^2 dx + \frac{1}{2}\int_{\T^d}\omega^\gamma|\Delta \varphi|^2 dx\\
%& \leq \frac{\gamma(2\gamma-2+d)}{4}\int_{\T^d} \omega^{\gamma-1} \varphi^2dx + \frac{1}{2}\int_{\T^d}\omega^\gamma|\Delta \varphi|^2 dx.
\end{split}
\end{align}
Now we use the following observation, as $|\nabla \omega|^2\leq \omega$, there holds
\begin{align*}
\Delta(\omega^\gamma) & = \gamma\omega^{\gamma-1}\Delta \omega +\gamma(\gamma-1)\omega^{\gamma-2}|\nabla\omega|^2 
\leq \gamma\omega^{\gamma-1}\Delta \omega +\gamma(\gamma-1)\omega^{\gamma-1}.
\end{align*}
Moreover, 
\begin{align*}
  \Big|
  \int_{\T^d}\omega^{\gamma-1}\Delta\omega\varphi^2 dx
  \Big| & = \Big|\int_{\T^d}\nabla\big(\omega^{\gamma-1}\varphi^2\big)\cdot\nabla\omega dx\Big|\\
  & \le (\gamma-1)
  \int_{\T^d}\omega^{\gamma-2}|\nabla\omega|^2\varphi^2 dx + 2\int_{\T^d}\omega^{\gamma-1}|\varphi||\nabla\omega||\nabla\varphi| dx\\
  & \le (\gamma-1) \int_{\T^d}\omega^{\gamma-1}\varphi^2 dx + 2\int_{\T^d}\omega^{\gamma-\frac{1}{2}}|\varphi||\nabla\varphi| dx\\
  & \le (\gamma-1) \int_{\T^d}\omega^{\gamma-1}\varphi^2 dx + 2\Big(\int_{\T^d}\omega^{\gamma-1}\varphi^2dx \Big)^\frac{1}{2}
  \Big(\int_{\T^d}\omega^\gamma|\nabla\varphi|^2 dx\Big)^\frac{1}{2},
\end{align*}
where we applied \eqref{estomega} and Cauchy-Schwarz inequality. By \eqref{general1}, we conclude that
\begin{align}
\label{new4}
  \Big|
  \int_{\T^d}\omega^{\gamma-1}\Delta\omega\varphi^2 dx
  \Big|
  \le
  \frac{C_\gamma}{\sqrt{d}}\Big(1+o_d(1)\Big)
  \int_{\T^d}\omega^\gamma |\nabla\varphi|^2 dx, \quad \forall\; \varphi \in \dot{C}^1(\T^d).
\end{align}
Coming back to \eqref{new1}, applying \eqref{new4}, \eqref{gap} and \eqref{general1}, there holds
\begin{align*}
%\label{new5}
\int_{\mathbb{T}^d} \omega^\gamma|\nabla\varphi|^2 dx & \le \Big(\frac{1}{2} +o_d(1)\Big)\int_{\mathbb{T}^d}\omega^\gamma{|\nabla\varphi|^2} dx + \frac{1}{2}\int_{\T^d}\omega^\gamma|\Delta \varphi|^2 dx,
\end{align*}
hence
\begin{align*}
\Big(\frac{1}{2}-o_d(1)\Big) \int_{\mathbb{T}^d} \omega^\gamma|\nabla\varphi|^2 dx \leq \frac{1}{2}\int_{\T^d}\omega^\gamma|\Delta \varphi|^2 dx
\end{align*}
which is equivalent to \eqref{gap1}. \qed

\medskip

\subsection{Proof of \eqref{general2}}
Finally, we consider \eqref{general2} with the Rellich identity \eqref{rellich}. Choose again $V =\omega^\gamma$, $f=\omega^\alpha$, we get
\begin{align*}
-\frac{{\rm div}(V\nabla f)}{f} - \Delta V & = -\frac{\alpha(\alpha+ \gamma -1)\omega^{\alpha + \gamma - 2}|\nabla \omega|^2 + \alpha\omega^{\alpha+\gamma-1}\Delta \omega}{\omega^\alpha} - \Delta(\omega^\gamma)\\
& = -\big[\alpha(\alpha+ \gamma -1)+\gamma(\gamma - 1)\big]|\nabla\omega|^2 \omega^{\gamma -2} -(\alpha + \gamma) \omega^{\gamma -1}\Delta \omega\\
& \geq -\Big[\alpha^2 + (\alpha + \gamma)\big(\gamma - 1 +\frac{d}{2}\big)\Big]\omega^{\gamma -1} + (\alpha + \gamma) \omega^{\gamma}.
\end{align*}
%Similarly, the singularity at the origin is again removed by replacing $\omega$ by $\omega+\epsilon$ and passing to the limit.

Furthermore,
\begin{align*}
V_{x_ix_j}=\gamma(\gamma-1)\omega^{\gamma-2} \omega_{x_i}\omega_{x_j} + \gamma \omega^{\gamma-1} \omega_{x_ix_j},
\end{align*}
as 
\begin{align*}
\omega_{x_i}=\frac{\sin x_i}{2},\quad  \omega_{x_ix_j} =\frac{\cos x_i}{2} \delta_{ij}.
\end{align*}
Thus, there hold
\begin{align*}
V_{x_ix_j} = \frac{\gamma(\gamma-1)}{4}\omega^{\gamma-2}\sin x_i\sin x_j + \frac{\gamma}{2}\omega^{\gamma -1}\cos x_i\delta_{ij}
\end{align*}
and
\begin{align*}
\sum_{i,j=1}^d\int_{\T^d}V_{x_ix_j}\varphi_{x_i}\varphi_{x_j}dx & = \frac{\gamma(\gamma-1)}{4}\int_{\T^d}\omega^{\gamma-2}\Big(\sum_{i=1}^d\sin(x_i)\varphi_{x_i}\Big)^2 dx\\
& \quad + \frac{\gamma}{2}\int_{\T^d}\omega^{\gamma-1}\sum_{i=1}^d\cos(x_i)\varphi_{x_i}^2 dx.
\end{align*}
Therefore, using 
\begin{align*}
\Big(\sum_{i=1}^d\sin(x_i)\varphi_{x_i}\Big)^2 \leq \Big(\sum_{i=1}^d \sin^2 x_i\Big)|\nabla \varphi|^2 \leq 4\omega(x)|\nabla \varphi|^2,
\end{align*}
we see that
\begin{align*}
\Big|\sum_{i,j=1}^d\int_{\T^d}V_{x_ix_j}\varphi_{x_i}\varphi_{x_j}dx\Big| & \leq \gamma(\gamma-1)\int_{\T^d}\omega^{\gamma-1}|\nabla\varphi|^2 dx + \frac{\gamma}{2}\int_{\T^d}\omega^{\gamma-1}|\nabla\varphi|^2 dx\\
& \leq \frac{2\gamma^2 - \gamma}{2} \int_{\mathbb{T}^d} \omega^{\gamma-1}|\nabla\varphi(x)|^2dx.
\end{align*}
By \eqref{rellich} and \eqref{gap1}, for $\alpha+\gamma<0$, we have
\begin{align}\label{g1}
\begin{split}
\int_{\T^d} \omega^\gamma|\Delta \varphi|^2 dx &\geq -\Big[\alpha^2 + (\alpha + \gamma)\big(\gamma - 1 +\frac{d}{2}\big)\Big]\int_{\mathbb{T}^d } \omega^{\gamma -1}|\nabla \varphi|^2 dx\\
& \quad + (\alpha + \gamma)\int_{\T^d}\omega^{\gamma}|\nabla \varphi|^2 dx - \frac{2\gamma^2 - \gamma}{2}\int_{\mathbb{T}^d}\omega^{\gamma-1}|\nabla\varphi|^2dx\\
& \ge -\Big[\alpha^2 + (\alpha + \gamma)\big(\gamma - 1 +\frac{d}{2}\big) + \frac{2\gamma^2 - \gamma}{2}\Big]\int_{\mathbb{T}^d } \omega^{\gamma -1}|\nabla \varphi|^2 dx\\
& \quad + (\alpha + \gamma)\big(1 + o_d(1)\big)\int_{\T^d}\omega^{\gamma}|\Delta \varphi|^2 dx.
\end{split}
\end{align}
Hence  
\begin{align*}
%\label{estg1}
\int_{\mathbb{T}^d} \omega^\gamma|\Delta \varphi|^2 dx \geq -\frac{G(\alpha)}{1-(\alpha+\gamma)(1+o_d(1))} \int_{\mathbb{T}^d }\omega^{\gamma-1}|\nabla \varphi(x)|^2 dx,
\end{align*}
where
\begin{align*}
G(t) = t^2 + (t + \gamma)\big(\gamma - 1 +\frac{d}{2}\big) + \frac{2\gamma^2 - \gamma}{2}.
\end{align*}
We check readily that, for any fixed $\gamma$, taking $\alpha = -\sqrt{d}$, 
\begin{align*}
G(-\sqrt d)=-\frac{1}{2} d^{3/2}+ (1+\frac{\gamma}{2})d-(\gamma-1)\sqrt d +2\gamma^2-\frac{3}{2}\gamma=-\frac{1}{2} d^{3/2}\Big(1+o_d(1)\Big).
\end{align*}
This means that
\begin{align*}
-\frac{G(-\sqrt d)}{1-(\gamma-\sqrt d)(1+o_d(1))}=\frac{d^{\frac{3}{2}}(1+o_d(1))}{2\sqrt d(1+o_d(1))}=\frac{d}{2}\Big(1+o_d(1)\Big).
\end{align*}
So we are done. \qed

\medskip

\section{Asymptotically sharp Hardy-Rellich inequalities}
\reset
Here we establish our main results, Theorem \ref{th3} and Theorem \ref{main}. 

\subsection{Proof of Theorem \ref{th3}} Consider the situation on $\T^d$, the lower bound estimates are direct consequences of \eqref{general1} and \eqref{general2}, which are equivalent to 
\begin{align*}
\liminf_{d\to\infty}\frac{C_{d,1,0,\gamma}}{d} \ge \frac{1}{2}, \quad \liminf_{d\to\infty}\frac{C_{d,2,1,\gamma}}{d} \ge \frac{1}{2}.
\end{align*}
 For every fixed $\eta\geq 1$ and every fixed $\ell\in\N$, applying \eqref{general1} and \eqref{general2}, there holds 
\begin{align}\label{k=l+1}
\liminf_{d\to\infty}\frac{C_{d,\ell+1,\ell,\eta}}{d} \ge \frac{1}{2}, \quad \forall\; \ell \in \N.
\end{align}
Indeed, if $\ell$ is even, we apply \eqref{general1} with $\Delta^{\frac{\ell}{2}}\phi$; whereas for $\ell$ odd,  we use \eqref{general2} to $\Delta^\frac{\ell-1}{2}\phi$.

\medskip
Let now $m=k-\ell$, then $m\leq \gamma$, we obtain by definition
\begin{align*}
C_{d,\ell+m,\ell,\gamma} \ge \prod_{i=\ell}^{\ell+m-1} C_{d,i+1,i,\gamma-m+1+i-\ell}, \quad \forall\; \ell \in \N.
\end{align*}

Applying \eqref{k=l+1}, we have
\begin{align}\label{lower}
\liminf_{d\to\infty}\frac{C_{d,\ell+m,\ell,\gamma}}{d^m} \ge 2^{-m}.
\end{align}

For the upper bound estimates, we choose a special test function in $C^\infty(\T^d)$. Let $\phi_1(x) = \sin(x_1)$. We claim that given any $\ell, \gamma \in \N$,
\begin{align}\label{phi1}
\int_{\T^d} |D^\ell\phi_1|^2\omega^\gamma dx = \Big(\frac{d}{2}\Big)^\gamma\frac{(2\pi)^d}{2} \Big[1 + O(d^{-1})\Big], \quad \mbox{as }\; d \rightarrow \infty.
\end{align}
Let $\ell \in 2\N$, so $D^\ell\phi_1 = (-1)^{\ell/2}\phi_1$. Recall that $\omega(x)= \frac{d}{2}(1-S)$ with $S:=\frac{1}{d}\sum_{1\le j\le d} \cos x_j$.
\begin{align*}
\Big(\frac{2}{d}\Big)^\gamma\int_{\T^d} |D^\ell\phi_1|^2\omega^\gamma dx & = \int_{\T^d} (\sin x_1)^2(1-S)^\gamma dx\\
& = \int_{\T^d} (\sin x_1)^2 dx + \sum_{k=1}^\gamma C_\gamma^k (-1)^k \int_{\T^d} (\sin x_1)^2 S^k dx\\
& = \frac{(2\pi)^d}{2} + \sum_{k=2}^\gamma C_\gamma^k (-1)^k \int_{\T^d} (\sin x_1)^2 S^k dx,
\end{align*}
where we used the fact that the integral of $(\sin x_1)^2S$ vanishes by periodicity and orthogonality. As $|S(x)| \le 1$, 
\begin{align*}
\Big|\sum_{k=2}^\gamma C_\gamma^k (-1)^k \int_{\T^d} (\sin x_1)^2 S^k dx\Big| \leq \sum_{k=2}^\gamma C_\gamma^k \int_{\T^d} S^2 dx \le 2^\gamma \int_{\T^d} S^2 dx.
\end{align*}
Moreover,
\begin{align*}
\int_{\T^d} S^2 dx = \frac{1}{d^2} \sum_{i=1}^d\sum_{j=1}^d \int_{\T^d} \cos x_i \cos x_j dx = \frac{1}{d^2} \sum_{i=1}^d \int_{\T^d} (\cos x_i)^2 dx =\frac{ (2\pi)^{d} }{2d},
\end{align*}
there holds
\begin{align*}
\int_{\T^d} |D^\ell\phi_1|^2\omega^\gamma dx & = \Big(\frac{d}{2}\Big)^\gamma\frac{(2\pi)^d}{2} \Big[1 + O(d^{-1})\Big],
\end{align*}
which implies \eqref{phi1} for $\ell$ even. 

\medskip
If $\ell \in 2\N + 1$ is odd, then
\begin{align*}
|D^\ell\phi_1|^2=(\cos x_1)^2
\end{align*}
and the same argument gives
\begin{align*}
\int_{\T^d} |D^\ell\phi_1|^2\omega^\gamma dx & = \Big(\frac{d}{2}\Big)^\gamma\frac{(2\pi)^d}{2} \Big[1 + O(d^{-1})\Big].
\end{align*}
So \eqref{phi1} is valid for all $\ell$. Applying \eqref{phi1}, there holds 
\begin{align}
\label{upper}
\limsup_{d\to\infty}\frac{C_{d,\ell+m,\ell,\gamma}}{d^m} \le 2^{-m}.
\end{align}
The proof of Theorem \ref{th3} is completed using \eqref{lower} and \eqref{upper}.\qed

\subsection{Proof of Theorem \ref{main}}

Finally, we handle the lattice case $\Z^d$. Combining Lemma \ref{g} and Theorem \ref{th3}, we see that for any $m \ge 1$,
\begin{align*}
\liminf_{d\to\infty}\frac{\C_{m,d}}{d^m} \geq 4^m\liminf_{d\to\infty}\frac{C_{d,2m,m,m}}{d^m} = 2^m.
\end{align*}

\medskip
Remark that the function $\Psi$ given by Lemma \ref{g} depends on the integer $m$ generally, so we cannot conclude the upper bound estimate for $\C_{m,d}$ obviously by Theorem \ref{th3}. By a closer inspection of the proof in \cite{G}, we observe that the function $\Psi$ yielding equality is given by 
\begin{align}
\label{Psim}
\Psi(x) = (-i)^m {\mathcal F}\Big\{\frac{u(n)}{|n|^{2m}}\Big\}(x).
\end{align}
Here ${\mathcal F}$ denotes the Fourier transform from $C_c(\Z^d)$ to $C^\infty(\T^d)$ given by
\begin{align*}
{\mathcal F}\{a(n)\}(x) = (2\pi)^{-d/2}\sum_{n\in \Z^d} a(n)e^{-in\cdot x}, \quad x \in \R^d.
\end{align*}

\smallskip
Take $$u_1(n) = \frac{(2\pi)^{d/2}}{2}\big(\delta_{e_1}-\delta_{-e_1}\big) \in \Ec(\Z^d)$$ with support in $\{|n|=1\}$, we see that the corresponding choice of $\Psi$, up to a constant factor, does not depend on $m$. Applying \eqref{Psim}, all the identities in Lemma \ref{g} hold with $\Psi$ replaced by a constant multiple of $\varphi_1(x)= \sin x_1$. Therefore, the estimates for the flat torus $\T^d$ yield
\begin{align*}
\limsup_{d\to\infty}\frac{\C_{m,d}}{d^m} \leq 2^m.
\end{align*}
This completes the proof.\qed

\medskip

\begin{rem}
After our work was completed, we learned that Gupta \cite{ Gupta} independently obtained the same result using a different approach.
\end{rem}

\noindent{\bf Acknowledgements.}
The authors are partially supported by NSFC (No.12271164) and Science and Technology Commission of Shanghai Municipality (No. 22DZ2229014).

\medskip

\noindent{\bf Data availability}
Data sharing is not applicable to this work as no new data were created or analyzed in the study.

\end{document}